\documentclass[11pt,a4paper,fullpage]{article}

\usepackage{amsmath,amssymb,amsthm}
\usepackage[latin2]{inputenc}
\usepackage[T1]{fontenc}

\usepackage[left=4cm,right=4cm, top=3cm, bottom=3cm, footskip=1.5cm,head=0cm]{geometry}

\usepackage{bigints}

\newcommand{\Ex }{\operatorname{E}}
\renewcommand{\P}{\operatorname{P}}
\newcommand{\Cov}{\operatorname{Cov}}

\newcommand{\w}[1]{\mathbf{#1}}
\newcommand{\ke}{\k{e}}

\makeatletter
\newcommand{\vast}{\bBigg@{3}}
\makeatother

\newtheorem{theorem}{Theorem}[section]
\newtheorem{corollary}[theorem]{Corollary}
\newtheorem{lemma}[theorem]{Lemma}
\newtheorem{remark}[theorem]{Remark}
\begin{document}
\title{\textbf{Extremes of multidimensional stationary\\ Gaussian random fields}}
\author{Natalia Soja-Kukie\l a%
\thanks{ Faculty of Mathematics and Computer Science, Nicolaus Copernicus University, ul.~Chopina 12/18, 87-100 Toru\'n, Poland; e-mail address: \texttt{natas@mat.umk.pl}}}
\date{}
\maketitle

\begin{abstract}
Let $\{X(\w{t}):\w{t}=(t_1, t_2, \ldots, t_d)\in[0,\infty)^d\}$ be a centered \mbox{stationary} Gaussian field with almost surely continuous sample paths, unit variance and correlation function $r$ satisfying conditions $r(\w{t})<1$ for \mbox{every $\w{t}\neq \w{0}$} and $r(\w{t})=1-\sum_{i=1}^d |t_i|^{\alpha_i} + o(\sum_{i=1}^d |t_i|^{\alpha_i})$, as $\w{t}\to\w{0}$, with constants $\alpha_1, \alpha_2, \ldots, \alpha_d \in(0,2]$. The main result of this contribution is the description of the~asymptotic behaviour of $\P(\sup\{X(\w{t}): \w{t}\in\mathcal{J}^{\w{x}}_{\w{m}} \}\leqslant u)$, as~$u\to\infty$, for some Jordan-measurable sets $\mathcal{J}^{\w{x}}_{\w{m}}$ of volume proportional to $\P(\sup\{X(\w{t}):\w{t}\in[0,1]^d\}>u)^{-1}(1+o(1))$.
\vspace{0.3cm}


\vspace{0.3cm}

\end{abstract}

\section{Introduction}\label{INTRO}
In extreme value theory of Gaussian processes, we have the following seminal result (see Leadbetter et al. \cite[Theorem 12.3.4]{LEAD}, Arendarczyk and D\k{e}bicki \cite[Lemma~4.3]{AREN},  Tan and Hashorva \cite[Lemma 3.3]{TAN}) concerning the~asymptotics of the distribution of supremum of a centered stationary Gaussian process $\{X(t):t\geqslant 0\}$ with correlation function satisfying
\begin{equation}\label{A1}
r(t)=\Cov(X(t),X(0))=1-|t|^\alpha+o(|t|^\alpha), \qquad \text{as} \quad t\to 0,
\end{equation}
for some $ \alpha\in(0,2]$, over intervals with length proportional to
$$\mu(u)=\P\left(\sup_{t\in[0,1]}X(t)>u\right)^{-1}(1+o(1)), \qquad\text{as}\quad u\to\infty.$$

\begin{theorem}\label{TW1WYM}
Let $\{X(t):t\geqslant 0\}$ be a zero-mean, unit-variance stationary Gaussian process with a.s. continuous sample paths and correlation function $r$ satisfying (\ref{A1}) and $r(t)\log t\to R \in [0,\infty)$ as $t\to\infty$. Let $0< A < B <\infty$.
Then
$$\P\left(\sup_{t\in[0,x\mu(u)]}X(t)\leqslant u\right)\to \Ex\exp\left(-x\exp\left(-R+\sqrt{2R}\mathcal{W}\right)\right),$$
as $u\to\infty$, uniformly for $x\in[ A, B]$, with $\mathcal{W}$ an $N(0,1)$ random variable.
\end{theorem}

It is natural to study a similar problem in the $d$-dimensional setting for arbitrary $d\in\mathbb{N}$. In this case one considers a centered stationary Gaussian process \mbox{$\{X(t_1,t_2,\ldots,t_d):t_1,t_2,\ldots,t_d\geqslant 0\}$} with unit variance and correlation function $r(t_1,t_2,\ldots,t_d)=\Cov(X(t_1,t_2,\ldots,t_d),X(0,0,\ldots,0))$ satisfying 
\begin{equation}\label{AA1}
r(t_1,t_2,\ldots,t_d) = 1-\sum_{i=1}^d|t_i|^{\alpha_i}\!+\!o\left(\sum_{i=1}^d|t_i|^{\alpha_i}\right),
\end{equation}
as $t_1,t_2,\ldots,t_d \to 0$, with $\alpha_1,\alpha_2,\ldots,\alpha_d\in(0,2]$.
The subject of interest is then the distribution of supremum of the field $\{X(t_1,t_2,\ldots,t_d)\}$ over sets of volume proportional to
$$m(u)=\P\left(\sup_{(t_1,t_2,\ldots,t_d)\in[0,1]^d}X(t_1,t_2,\ldots,t_d)>u\right)^{-1}(1+o(1)).$$
In this paper we investigate suprema over sets of the form
\begin{equation*}
\mathcal{J}^{\w{x}}_\w{m}:=\left\{(t_1,t_2,\ldots,t_d)\in\mathbb{R}^d:\left(\frac{t_1}{x_1m_1(u)},\frac{t_2}{x_2m_2(u)},\ldots,\frac{t_d}{x_dm_d(u)}\right)\in\mathcal{J}\right\},
\end{equation*}
where $\mathcal{J}\subset\mathbb{R}^d$
is a Jordan-measurable set with Lebesgue measure ${\lambda}(\mathcal{J})>0$, \mbox{$\w{x}=(x_1,x_2,\ldots,x_d)\in(0,\infty)^d$}  and $\w{m}=(m_1,m_2,\ldots, m_d)$ with $m_1, m_2\ldots,m_d$ some positive functions satisfying $m_1(u)m_2(u)\cdots m_d(u)=m(u)$. We denote $\mathcal{J}_{\w{m}}:=\mathcal{J}_{\w{m}}^{(1,\ldots,1)}$.
One interesting case is $\mathcal{J}=[0,1]^d$ with $\mathcal{J}_\w{m}^{\w{x}}=\prod_{i=1}^d[0,x_im_i(u)]$.

In a recent paper D\k{e}bicki et al. \cite{DEB} consider the case $d=2$. They assume that the functions $m_1$ and $m_2$ tend to infinity and satisfy
\begin{equation}\label{OGR}
\frac{\log m_1(u)}{\log m_2(u)}\to 1,\qquad\text{as}\quad u\to\infty.
\end{equation}
The authors establish the following $2$-dimensional counterpart \cite[Theorem 2]{DEB} of Theorem \ref{TW1WYM}.
\begin{theorem}\label{TW2WYM}
Let $\{X(t_1,t_2): t_1,t_2\geqslant 0\}$ be a zero-mean, unit-variance stationary Gaussian field with a.s. continuous sample paths and correlation function $r$ satisfying (\ref{AA1}) and
$ r(t_1,t_2)\log \sqrt{t_1^2+t_2^2}\to R\in [0,\infty)$ as $t_1^2+t_2^2\to\infty$.
Let $m_1$ and $m_2$ be positive functions such that $m_1(u)m_2(u)=m(u)$ and (\ref{OGR}) hold. Then:\\
(i)\; for each $0< A < B <\infty$,
$$\P\left(\sup_{(t_1,t_2)\in\left[0,x_1m_1\right]\times\left[0,x_2m_2\right]}X(t_1,t_2)\leqslant u\right)
\to
\Ex e^{-x_1x_2\exp(-2R+2\sqrt{R}\mathcal{W})},$$
as $u\to\infty$, uniformly for $(x_1,x_2)\in[A,B]^2$, with $\mathcal{W}$ an $N(0,1)$ random variable;\\
(ii)\; for every Jordan-measurable set $\mathcal{J}\subset\mathbb{R}^2$ with Lebesgue measure ${\lambda}(\mathcal{J})>0$, 
\[
\P\left(\sup_{(t_1,t_2)\in \mathcal{J}_\w{m}}X(t_1,t_2)\leqslant u\right)
\to
\Ex e^{- {\lambda}(\mathcal{J})\exp(-2R+2\sqrt{R}\mathcal{W})},
\]
as $u\to\infty$, with $\mathcal{W}$ an $N(0,1)$ random variable.
\end{theorem}

Our goal is to derive a general limit theorem for the distribution of supremum of the field $\{X(t_1,t_2,\ldots,t_d)\}$ over sets $\mathcal{J}^{\w{x}}_\w{m}$, for arbitrary $d\in\mathbb{N}$ and for a~wide class of families $\{m_1, m_2, \ldots, m_d\}$ of functions, uniform for $\w{x}\in[A,B]^d$, for all $0<A<B<\infty$. The main result is Theorem \ref{TWDWYM}. In the paper we do not assume that every $m_i$ tends to infinity like D\k{e}bicki et al. \cite{DEB} do. We fully explain the case when all $m_i$s are separated from zero (see Theorem \ref{TWDWYM} and Remark \ref{MORE}) and give some partial results in the case when some of $m_i$s tend to zero (see Corollaries \ref{SZYBKO} and~\ref{WOLNO}).

\section{Preliminaries}\label{PRE}

We consider $\mathbb{R}^d$ with coordinatewise order $\leqslant$, write $\w{t}=(t_1,t_2,\ldots,t_d)$ for an element $\w{t}\in\mathbb{R}^d$, put  $\w{0}:=(0,0,\ldots,0)$ and $\w{1}:=(1,1,\ldots,1)$, and denote by $\|\cdot\|_{\infty}$ the~sup-norm in $\mathbb{R}^d$, i.e., $\|\w{t}\|_{\infty}=\max\{|t_1|,|t_2|,\ldots,|t_d|\}$ for any $\w{t}\in\mathbb{R}^d$.

Let $\{X(\w{t}):\w{t}\in[0,\infty)^d\}$ be a centered stationary Gaussian field 
with a.s. continuous sample paths, unit variance and correlation function 
$$r(\w{t})=\Cov(X(\w{t}),X(\w{0})).$$
We will often assume that the correlation function satisfies:\\
\textbf{A1}:
$r(\w{t})=1-\sum_{i=1}^d|t_i|^{\alpha_i}+o\left(\sum_{i=1}^d|t_i|^{\alpha_i}\right),$
as $t_1,t_2,\ldots,t_d\!\to\! 0$;\\
\textbf{A2}:
$r(\w{t})<1$ for $\w{t}\neq \w{0}$;\\
\textbf{A3}:
$r(\w{t})\log \sqrt{t_1^2+t_2^2+\ldots+t_d^2} \to R$, as $t_1^2+t_2^2+\ldots+t_d^2\to\infty$,\\
with some constants $\alpha_1,\alpha_2,\ldots ,\alpha_d \in(0,2]$ and $R\in[0,\infty)$. The above conditions are analogous to the ones given in \cite{LEAD,AREN,TAN,DEB}.

Condition \textbf{A1} implies that the correlation function $r$ is continuous. \textbf{A1} and \textbf{A2} give $|r(\w{t})|<1$ for $\w{t}\neq \w{0}$. Moreover, condition \textbf{A2} follows from \textbf{A1} and \textbf{A3}. Notice that we study both {\it weakly dependent} fields, satisfying \textbf{A3} with $R=0$, and {\it strongly dependent} fields, satisfying \textbf{A3} with~$R\in(0,\infty)$.

For every $\alpha\in(0,2]$, we denote by $\mathcal{H}_\alpha$ the Pickands constant (see \cite{PICKANDS}), i.e.,
$$\mathcal{H}_\alpha:=
\lim_{T\to\infty}\frac{\Ex \exp\left(\max_{0\leqslant t\leqslant T} B_{\alpha/2}(t)-|t|^{\alpha} \right)}{T},$$
where $\{B_{\alpha/2}(t):t\geqslant 0\}$ is a fractional Brownian motion with Hurst index $\alpha/2$. 

Let $\mathcal{W}$ be a standard normal random variable and let $\Phi(u):=P(\mathcal{W}\leqslant u)$,
\mbox{$\Psi(u):=\P(\mathcal{W}>u)$}.
We recall that
$$\Psi(u)=\frac{1}{\sqrt{2\pi}u}\exp\left(-\frac{u^2}{2}\right)(1+o(1))
\qquad\text{as}\quad u\to\infty.$$

If the considered field $\{X(\w{t})\}$ satisfies \textbf{A1} and \textbf{A2}, then, for arbitrary Jordan-measurable set $\mathcal{J}\subset\mathbb{R}^d$ with Lebesgue measure ${\lambda}(\mathcal{J})>0$, we have
\begin{equation}\label{ASYMPTOTICS}
\P\left(\max_{\w{t}\in\mathcal{J}}X(\w{t})>u\right)=\lambda (\mathcal{J})\prod_{i=1}^d \left(\mathcal{H}_{\alpha_i} u^{2/\alpha_i}\right)\Psi(u)(1+o(1)),
\end{equation}
as $u\to\infty$, due to  Piterbarg \cite[Theorem 7.1]{PITERBARG}. Thus
\begin{equation*}
m(u):=\left(\prod_{i=1}^d \left(\mathcal{H}_{\alpha_i} u^{2/\alpha_i}\right)\Psi(u)\right)^{-1} = \P\left(\max_{\w{t}\in[0,1]^d}X(\w{t})>u\right)^{-1}(1+o(1)).
\end{equation*}

Let $m_1,m_2,\ldots,m_d$ be positive functions such that
$$m_1(u)m_2(u)\cdots m_d(u)=m(u)$$
and  for some $k\in\{0,1,\ldots, d-1\}$:\\
1.\; for every $i\in\{1,2,\ldots,k\}$ there exists an $M_i\in(0,\infty)$ such that
$$m_i(u)\to M_i \quad\text{as}\quad u\to\infty;$$
2.\; for every $i\in\{k+1,k+2,\ldots,d\}$ we have 
$$m_i(u)\to\infty\qquad\text{and}\qquad m_i(u)=\exp(\gamma_i u^2)c_i(u),\qquad\text{as}\quad u\to\infty,$$ 
for some constant $\gamma_i\in[0,1/2]$ and positive function $c_i$ with \mbox{$\log c_i(u)=o(u^2)$}.
Then $\gamma_{k+1}+\gamma_{k+2}+\ldots+\gamma_d=1/2$. We put $\gamma:=\max_{i}\gamma_i$.

For arbitrary $\w{x}\in(0,\infty)^d$, we define
$\mathcal{R}^{\w{x}}:=[0,x_1]\times[0,x_2]\times\cdots\times[0,x_d]$ and
$\mathcal{R}^{\w{x}}_\w{m}:=[0,x_1m_1(u)]\times[0,x_2m_2(u)]\times\cdots\times[0,x_dm_d(u)]$ for each $u\in\mathbb{R}$. Note that $\mathcal{R}^{\w{x}}_\w{m}=\mathcal{J}^{\w{x}}_\w{m}$ for $\mathcal{J}=[0,1]^d$.

\section{Results}\label{RESULTS}

Below, in Section \ref{MAIN_T}, we present Theorem \ref{TWDWYM}, which is the main result. Its proof is given in Sections \ref{LEMMAS} and \ref{DWYM_DD}.  Some consequences of Theorem \ref{TWDWYM} can be found in Sections \ref{MAIN_T} and \ref{CONS}.

\subsection{Main theorem}\label{MAIN_T}
The following theorem describes the asymptotic behaviour of $$\P\left(\sup\{X(\w{t}):\w{t}\in\mathcal{J}^{\w{x}}_\w{m}\}\leqslant u\right),$$ as $u\to\infty$, for Jordan-measurable sets $\mathcal{J}_\w{m}^{\w{x}}$  of volume proportional to $m(u)$.

\begin{theorem}\label{TWDWYM}
Let $\{X(\w{t}):\w{t}\in[0,\infty)^d\}$ be a centered stationary Gaussian field with a.s. continuous sample paths, unit variance and correlation function $r$ that satisfies \textbf{A1} and \textbf{A3} with some $R\in[0,\infty)$. Then, for every Jordan-measurable set $\mathcal{J}\subset\mathbb{R}^d$ with ${\lambda}(\mathcal{J})>0$, for each $0< A < B <\infty$, 
\[
\P\left(\sup_{\w{t}\in \mathcal{J}^{\w{x}}_\w{m}}X(\w{t})\leqslant u\right)
\to
\Ex\exp\left(-x_1x_2\cdots x_d {\lambda}(\mathcal{J})\exp\left(-\frac{R}{2\gamma}+\sqrt{\frac{R}{\gamma}}\mathcal{W}\right)\right),
\]
as $u\to\infty$, uniformly for $\w{x}\in [A , B]^d$.
\end{theorem}

Applying the above theorem for $\mathcal{J}=[0,1]^d$, we obtain the following result.
\begin{corollary}
Let $\{X(\w{t})\}$ satisfy the assumptions of Theorem \ref{TWDWYM}. Then, for each $0< A < B <\infty$,
$$\P\left(\sup_{\w{t}\in\mathcal{R}^{\w{x}}_\w{m}}X(\w{t})\leqslant u\right)
\to
\Ex\exp\left(-x_1x_2\cdots x_d \exp\left(-\frac{R}{2\gamma}+\sqrt{\frac{R}{\gamma}}\mathcal{W}\right)\right),$$
as $u\to\infty$, uniformly for $\w{x}\in [A , B]^d$.
\end{corollary}

In the special case, when $k=0$ and the~functions $m_1, m_2, \ldots, m_d$ are chosen so that $\gamma_1=\gamma_2=\ldots=\gamma_d$ (and thus a $d$-dimensional analog of~(\ref{OGR}) holds), we have the following corollary. Note that for $d=2$ it coincides with Theorem~\ref{TW2WYM}.

\begin{corollary}
Let the assumptions of Theorem \ref{TWDWYM} be satisfied and let
\begin{equation}\label{OGR2}
\frac{\log m_i(u)}{\log m_j(u)}\to 1 \qquad\text{as}\quad u\to\infty,\qquad \text{for}\quad i,j\in\{1,2,\ldots,d\}.
\end{equation}
Then, for every Jordan-measurable set $\mathcal{J}\subset\mathbb{R}^d$ with 
${\lambda}(\mathcal{J})>0$,
\[
\P\left(\sup_{\w{t}\in \mathcal{J}^{\w{x}}_\w{m}}X(\w{t})\leqslant u\right)
\to
\Ex\exp\left(- x_1x_2\cdots x_d{\lambda}(\mathcal{J})\exp\left(-dR+\sqrt{2dR}\mathcal{W}\right)\right),
\]
as $u\to\infty$,  uniformly for $\w{x}\in [A , B]^d$, for each $0<A<B<\infty$.
\end{corollary}

\subsection{Some consequences of the main theorem}\label{CONS}
Let the field $\{X(\w{t})\}$ satisfy the assumptions of Theorem \ref{TWDWYM}. In this section we ask for the asymptotic behaviour of the supremum of~$\{X(\w{t})\}$ over sets
$\mathcal{J}^{\w{x}}_{\bar{\w{m}}}$, for $\mathcal{J}\subset\mathbb{R}^d$ a~Jordan-measurable set with  
${\lambda}(\mathcal{J})>0$, $\w{x}\in(0,\infty)^d$,
$\bar{\w{m}}=(\bar{m}_1, \bar{m}_2, \ldots, \bar{m}_d)$
and $\bar{m}_1, \bar{m}_2, \ldots, \bar{m}_d$ some positive functions satisfying $\bar{m}_1(u) \bar{m}_2(u) \cdots \bar{m}_d(u)=m(u)$.
We do not assume that $\bar{m}_1, \bar{m}_2, \ldots, \bar{m}_d$ fulfill all the conditions, which have to be satisfied by the functions $m_1, m_2, \ldots, m_d$ introduced in Section \ref{PRE}.

First, we consider the case when the functions $\bar{m}_1, \bar{m}_2, \ldots, \bar{m}_d$ are separated from zero, i.e., $\bar{m}_1(u), \bar{m}_2(u), \ldots, \bar{m}_d(u) > \varepsilon$ for some $\varepsilon>0$.
Then, it is easy to show, that every sequence $\{u_n\}_{n\in\mathbb{N}}$ tending to infinity contains a subsequence $\{u_{n_j}\}_{j\in\mathbb{N}}$ such that for each $i\in\{1,2,\ldots,d\}$ we have $\bar{m}_i(u_{n_j})\to \bar{M}_i\in[\varepsilon,\infty)$, as $j\to\infty$, or, alternatively,  $\bar{m}_i(u_{n_j})=\exp(\bar{\gamma}_i u_{n_j}^2)\bar{c}_i(u_{n_j})\to\infty$, as $j\to\infty$, for some constant $\bar{\gamma}_i \in[0,1/2]$ and some function $\bar{c}_i$ with $\log \bar{c}_i(u_{n_j})=o(u_{n_j}^2)$. We can apply Theorem~\ref{TWDWYM} for such subsequences. 
This justifies the following remark.
\begin{remark}\label{MORE}
Theorem \ref{TWDWYM} fully explains the case when $\bar{m}_1, \bar{m}_2, \ldots, \bar{m}_d$ are positive functions separated from zero, such that $\bar{m}_1(u) \bar{m}_2(u) \cdots \bar{m}_d(u)=m(u)$. It gives the asymptotics for convergent subsequences.
\end{remark}

Since for weakly dependent Gaussian fields the limit in Theorem \ref{TWDWYM} does not depend on $\gamma$, the above considerations entail a concise corollary.
\begin{corollary}
Let $\{X(\w{t})\}$ satisfy the assumptions of Theorem \ref{TWDWYM} with $R=0$ and let $\bar{m}_1, \bar{m}_2,\ldots, \bar{m}_d$ be positive functions separated from zero, such that $\bar{m}_1(u) \bar{m}_2(u) \cdots \bar{m}_d(u) = m(u)$.
Then, for each $0<A<B<\infty$,
\[
\P\left(\sup_{\w{t}\in \mathcal{J}^{\w{x}}_{\bar{\w{m}}}}X(\w{t})\leqslant u\right)
\to
\exp (- x_1x_2\cdots x_d{\lambda}(\mathcal{J})),
\]
as $u\to\infty$,  uniformly for $\w{x}\in [A , B]^d$.
\end{corollary}

Next, we focus on the case when $\bar{m}_i$s are allowed to tend to zero. In general, such weakening of the assumptions enforces a~different approach. However, basing on Theorem \ref{TWDWYM}, we can give the limit theorems in two special opposite cases: when $\bar{m}_i\to 0$ sufficiently fast and when $\bar{m}_i\to 0$ sufficiently slow.

Suppose that for some $0\leqslant j \leqslant k < d$:\\
0.\; for every $i\in\{1,2,\ldots, j\}$ we have
$$\bar{m}_i(u)\to 0\qquad \text{as} \quad u\to\infty;$$
1.\; for every $i\in\{j+1,j+2,\ldots, k\}$ there exists an $\bar{M}_i\in(0,\infty)$ such that
$$\bar{m}_i(u)\to \bar{M}_i \qquad \text{as} \quad u\to\infty;$$  
2.\; for every $i\in\{k+1,k+2,\ldots, d\}$
$$\bar{m}_i(u)\to \infty \qquad\text{and}\qquad \bar{m}_i(u)=\exp(\bar{\gamma}_i u^2)\bar{c}_i(u),\qquad \text{as} \quad u\to\infty,$$
hold for some constant $\bar{\gamma}_i\geqslant 0$ and function $\bar{c}_i$ such that $\log \bar{c}_i(u)=o(u^2)$.
Then $\bar{\gamma}_{k+1}+\bar{\gamma}_{k+2}+\ldots+\bar{\gamma}_d\geqslant 1/2$. We put $\bar{\gamma}:=\max_{i}\bar{\gamma}_i$.

Note that the above conditions are very similar to the conditions given in Section \ref{PRE} for the functions $m_1, m_2, \ldots, m_d$. Under these  assumptions (and some extra ones) we can prove the following results.

\begin{corollary}\label{SZYBKO}
Assume that $\bar{m}_1, \bar{m}_2,\ldots, \bar{m}_d$ satisfy the above conditions and, moreover,
$$\bar{m}_1(u)=\exp(-\kappa u^2) c (u)$$ for some constant $\kappa>0$ and function $c$ satisfying $\log c(u)=o(u^2)$.
Then, 
$$\P\left(\sup_{\w{t}\in \mathcal{J}^{\w{x}}_{\bar{\w{m}}} }X(\w{t})\leqslant u\right) \to 0,\qquad\text{as}\quad u\to\infty,$$
uniformly for $\w{x}\in [A , \infty)^d$, for each $A>0$.
\end{corollary}
\begin{proof}
Let $\w{x}\in(0,\infty)^d$. Since the set $\mathcal{J}\subset\mathbb{R}^d$ is Jordan-measurable and $\lambda(\mathcal{J})>0$, there exist $\w{y}\in\mathbb{R}^d$ and $\w{z}\in(0,\infty)^d$ such that $\w{y}+\mathcal{R}^{\w{z}}\subset \mathcal{J}$.
Thus
$$\P\left(\sup_{\w{t}\in \mathcal{J}^{\w{x}}_{\bar{\w{m}}} }X(\w{t})\leqslant u\right)\leqslant \P\left(\sup_{\w{t}\in (\w{y} + \mathcal{R}^{\w{z}})^{\w{x}}_{\bar{\w{m}}} }X(\w{t})\leqslant u\right) =
\P\left(\sup_{\w{t}\in \mathcal{R}^{\w{zx}}_{\bar{\w{m}}}}X(\w{t})\leqslant u\right),$$
with $\w{zx}:=(z_1x_1, z_2x_2,\ldots,z_dx_d)$, where the last equality is a consequence of stationarity. Furthermore,
\begin{equation*}
\P\left(\sup_{\w{t}\in\mathcal{R}^{\w{zx}}_{\bar{\w{m}}}}X(\w{t})\leqslant u\right)
\leqslant \P\left(\sup_{0\leqslant t_i\leqslant z_ix_i\bar{m}_i }X(0,\ldots,0,t_{k+1},t_{k+2},\ldots,t_d)\leqslant u\right).
\end{equation*}
We will show, the right-hand side of the above inequality tends to zero, \mbox{applying} Theorem \ref{TWDWYM} for the field $\hat{X} (t_{k+1},t_{k+2}\ldots ,t_d):=X(0,\ldots,0,t_{k+1},t_{k+2},\ldots,t_d)$, $t_{k+1},t_{k+2},\ldots ,t_d \geqslant 0$, that satisfies $(d-k)$-dimensional conditions \textbf{A1} and \textbf{A3}. 

Since $\kappa>0$, we have $\sigma:=\bar{\gamma}_{k+1}+\bar{\gamma}_{k+2}+\ldots +\bar{\gamma}_{d}>1/2$. Hence
$$ \frac{\bar{m}_{k+1}(u)\bar{m}_{k+2}(u)\cdots \bar{m}_d(u)}{\hat{m}(u)} \to \infty ,\qquad \text{as}\quad u\to \infty,$$
where 
$$\hat{m}(u):=\left(\prod_{i=k+1}^d \left(\mathcal{H}_{\alpha_i} u^{2/\alpha_i}\right)\Psi(u)\right)^{-1}.$$
For every $i\in\{k+1,k+2,\ldots,d\}$, we put 
$$\hat{m}_i(u):=\exp(\hat{\gamma}_i u^2) \hat{c}_i(u),$$
with $\hat{\gamma}_i:=(2\sigma)^{-1}\bar{\gamma}_i$ and $\hat{c}_i(u):=(\hat{m}(u)\exp(-u^2/2))^{1/(d-k)}$. Then $\hat{\gamma}_i \in [0,1/2]$, $\log \hat{c}_i(u)=o(u^2)$ and $\hat{\gamma}_{k+1}+\hat{\gamma}_{k+2}+ \ldots + \hat{\gamma}_d = 1/2$. Moreover,
the functions $\hat{m}_i$ satisfy $\hat{m}_{k+1}(u) \hat{m}_{k+2}(u) \cdots \hat{m}_d(u) = \hat{m}(u)$
and we have
$$\frac{\bar{m}_i(u)}{\hat{m}_i(u)}\to \infty \qquad \text{as} \quad u\to \infty.$$

Let $C>0$ be arbitrary. 
Since $\tilde{m}_i(u) / \hat{m}_i(u) > C$ for all sufficiently large $u$, we obtain
\begin{eqnarray*}
\lefteqn{\limsup_{u\to\infty}\P\left(\sup_{0\leqslant t_i\leqslant x_i\tilde{m}_i}X(0,\ldots,0,t_{k+1},t_{k+2},\ldots,t_d)\leqslant u\right)}\nonumber\\
&=&\limsup_{u\to\infty}\P\left(\sup_{0\leqslant t_i\leqslant x_i\tilde{m}_i }\hat{X}(t_{k+1},t_{k+2},\ldots,t_d)\leqslant u\right)\nonumber\\
&\leqslant& \limsup_{u\to\infty}\P\left(\sup_{0\leqslant t_i\leqslant C x_i \hat{m}_i}\hat{X}(t_{k+1},t_{k+2},\ldots,t_d)\leqslant u\right)\nonumber\\
&=&\Ex\exp\left(-C^{d-k}x_{k+1}x_{k+2}\cdots x_d \exp\left(-\frac{R}{2\hat{\gamma}}+\sqrt{\frac{R}{\hat{\gamma}}}\mathcal{W}\right)\right),
\end{eqnarray*}
with $\hat{\gamma}:=\max_i \hat{\gamma_i}$, due to Theorem \ref{TWDWYM}. Since the right-hand side tends to zero as $C\to\infty$, the proof of pointwise convergence is complete. Uniform convergence simply follows from the monoticity of $\w{x}\mapsto \P\left(\sup \{X(\w{t})\leqslant u : \w{t}\in \mathcal{J}^{\w{x}}_{\bar{\w{m}}}\}\right)$.
\end{proof}

\begin{corollary}\label{WOLNO}
Let $m_1,m_2,\ldots, m_d$ be the functions from Section \ref{PRE} and, moreover, assume that $m_i\equiv 1$ for $i\in\{1,2,\ldots, j\}$.
There exist some positive functions $\nu_1, \nu_2, \ldots, \nu_j$ satisfying $\nu_i(u)\to 0$, such that for all  $\bar{m}_1, \bar{m}_2, \ldots, \bar{m}_d$
satisfying $\nu_i(u)=o(\bar{m}_i(u))$ for each $i\in\{1,2,\ldots, j\}$, $\bar{m}_i(u)=m_i(u)$ for each $i\in\{j+1, j+2, \ldots, d-1\}$ and $\bar{m}_d(u)=m_d(u) \cdot \prod_{i=1}^j \bar{m}_i(u)^{-1}$, we have
$$\P\left(\sup_{\w{t}\in\mathcal{J}_{\bar{\w{m}}}}X(\w{t})\leqslant u\right)
\to
\Ex\exp\left(-\lambda(\mathcal{J}) \exp\left(-\frac{R}{2\gamma}+\sqrt{\frac{R}{\gamma}}\mathcal{W}\right)\right),$$
as $u\to\infty$.
\end{corollary}
\begin{proof}
Let $\varepsilon_1, \varepsilon_2, \ldots, \varepsilon_j > 0$ and $\pmb{\varepsilon}:=(\varepsilon_1, \varepsilon_2, \ldots, \varepsilon_j, 1,\ldots,1, \prod_{i=1}^j \varepsilon_i^{-1})$. By application of Theorem~\ref{TWDWYM}, we obtain that
$$\P\left(\sup_{\w{t}\in \mathcal{J}_{\w{m}}^{\pmb{\varepsilon}}}X(\w{t})\leqslant u\right)
\to \Ex \exp\left(-\lambda(\mathcal{J})\exp\left(-\frac{R}{2\gamma}+\sqrt{\frac{R}{\gamma}}\mathcal{W}\right)\right),$$
as $u\to\infty$, uniformly for $\pmb{\varepsilon}\in[A,B]^d$, for all $0<A<B<\infty$.
Note that the above limit does not depend on the choice of $\varepsilon_1, \varepsilon_2, \ldots, \varepsilon_j $.  It is not difficult to show that there exist some functions $\nu_i$, $i\in\{1,2,\ldots,j\}$, tending to zero, such that for positive functions $\varepsilon_i=\varepsilon_i(u)$, $i\in\{1,2,\ldots,j\}$, tending to zero, and for $\pmb{\varepsilon}(u):=(\varepsilon_1(u), \varepsilon_2(u), \ldots, \varepsilon_j(u),1,\ldots,1, \prod_{i=1}^j \varepsilon_i(u)^{-1})$, we have
$$\P\left(\sup_{{\w{t}\in \mathcal{J}^{\pmb{\varepsilon}(u)}_{\w{m}}}}  X(\w{t})\leqslant u\right)
\to \Ex\exp\left(-\lambda(\mathcal{J})\exp\left(-\frac{R}{2\gamma}  +  \sqrt{\frac{R}{\gamma}}\mathcal{W}\right)\right),$$
whenever
$\nu_i(u)=o(\varepsilon_i(u))$. We shall put $\varepsilon_i(u)=\bar{m}_i(u)$ for $i\in\{1,2,\ldots,j\}$.
\end{proof}

\begin{remark}
We do not know the form of the functions $\nu_1, \nu_2, \ldots, \nu_j$ from Corollary \ref{WOLNO}. Our conjecture is that $\nu_i(u)=u^{-2/\alpha_i}$ for $i\in\{1,2,\ldots, j\}$.
\end{remark}

\subsection{Lemmas}\label{LEMMAS}

The lemmas formulated in this section are crucial in the proof of Theorem \ref{TWDWYM} (see Section \ref{DWYM_DD}). They are $d$-dimensional counterparts of known results: Lemma \ref{LEMAT1} generalizes \cite[Lemma~12.2.11]{LEAD} and \cite[Lemma 1]{DEB}; Lemma~\ref{LEMACIK2}
combines $d$-dimensional analogs of \cite[Lemma 12.3.1]{LEAD} (for weakly dependent fields) and \cite[Lemma 3.1]{TAN} (for strongly dependent fields), it is a generalization of \cite[Lemma 2]{DEB}. Since the argumentation for Lemmas \ref{LEMAT1} and \ref{LEMACIK1} mimics the one given in~\cite{LEAD} and expanded in \cite{TAN,DEB2}, the proofs are skipped. We present the proof of Lemma \ref{LEMACIK2}, which significantly improves the lemma given by D\k{e}bicki et al. \cite{DEB,DEB2} and enables us to establish far more general results than the ones in \cite {DEB}.

Let $a>0$. Put $q_i=q_i(u):=au^{-2/\alpha_i}$ for $i\in\{1,2,\ldots,d\}$. Moreover, define $\w{jq}=\w{jq}(u):=(j_1q_1(u), j_2q_2(u), \ldots, j_dq_d(u))$ for $\w{j}=(j_1,j_2,\ldots, j_d) \in\mathbb{Z}^d$.

\begin{lemma}\label{LEMAT1}
Assume that conditions \textbf{A1} and \textbf{A2} hold. 
Then there exists a~function $\vartheta$ satisfying $\vartheta(a)\to 0$, as $a\to 0$, such that for every $a>0$ we have
$$\P\! \left(\sup_{\w{jq}\in \w{y}+\mathcal{R}^{\w{x}}} \!\!\! X(\w{jq})\leqslant u\right)-\P\! \left(\sup_{\w{t}\in \w{y}+\mathcal{R}^{\w{x}}} \!\!\! X(\w{t})\leqslant u\right)\leqslant \frac{x_1x_2\cdots x_d}{m}\vartheta(a)+o\left(\frac{1}{m}\right),$$
as $u\to\infty$, uniformly for $\w{y}\in [0,\infty)^d$ and $\w{x}\in[A,B]^d$, for all $0<A<B<\infty$.
\end{lemma}

\begin{remark}
An explicit formula for $\vartheta$ from Lemma \ref{LEMAT1} can be found in \cite{DEB2}.
\end{remark}

\begin{lemma}\label{LEMACIK1}
Suppose that $T=T(u)\to\infty$ as $u\to\infty$.
Then, providing that conditions \textbf{A1}~and~\textbf{A2} are fulfilled, there exists an $\varepsilon>0$ such that for all $R\geqslant 0$
\begin{eqnarray*}
\frac{m}{q_1q_2\cdots q_d} \! \sum_{\substack{\w{jq}\in (-\varepsilon,\varepsilon)^d\\ \w{jq}\neq(0,0,\ldots,0)}} \!
\vast[(1 \! - \! r(\w{jq}))\frac{R}{\log T} \left(1 \! - \! \left(r(\w{jq}) \! + \! (1 \! - \! r(\w{jq}))
\frac{R}{\log T}\right)^{\!2}\right)^{\!-1/2}\\ 
\times\exp \left(-\frac{u^2}{1 \! + \! r(\w{jq}) \! + \! (1 \! - \! r(\w{jq})) R / \log T}\right)\vast]\to 0,
\end{eqnarray*}
as $u\to\infty$.
\end{lemma}

Let $R\geqslant 0$ be fixed. The last lemma concerns functions $\rho_T$ and $\varrho_T$ defined for an arbitrary $T>1$ and for $\w{t}\in\mathbb{R}^d$ as follows:
\begin{eqnarray}\label{FUN}
\hspace{0.4cm}\rho_T(\w{t}) := \left\{ \begin{array}{ll}
1
&, \, \max\{|t_{k+1}|, |t_{k+2}|, \ldots, |t_d|\} < 1;\\
\left|r(\w{t})-\frac{R}{\log T}\right|
&,\, \text{otherwise},\end{array} \right.
\end{eqnarray}
\begin{eqnarray*}
\hspace{1cm}\varrho_T(\w{t}) := \left\{ \begin{array}{ll}
|r(\w{t})|+(1-r(\w{t}))\frac{R}{\log T}
&, \, \max\{|t_{k+1}|, |t_{k+2}|, \ldots, |t_d|\} < 1;\\
\frac{R}{\log T}
&, \, \text{otherwise}.\end{array} \right.
\end{eqnarray*}

\begin{lemma}\label{LEMACIK2}
Assume that $T_i=T_i(u)\sim\tau_i m_i(u)$, as~$u\to\infty$, for some $\tau_i>0$ and every $i\in\{1,2,\ldots,d\}$. Let $\varepsilon>0$. 
Then, providing that conditions \textbf{A1} and \textbf{A3} with $R\in[0,\infty)$ are fulfilled,
$$
\frac{T_1T_2\cdots T_d}{q_1q_2\cdots q_d} 
\! \sum_{\substack{\w{jq}\in \prod_{i=1}^d[-T_i,T_i]\\ \w{jq}\notin (-\varepsilon,\varepsilon)^d}} \!
\rho_{T} (\w{jq})\exp \! \left(-\frac{u^2}{1\!+\!\max\{|r(\w{jq})|,\varrho_{T} (\w{jq})\}}\right)\to 0,
$$
as $u\to\infty$, with $T:=\max\{T_1,T_2,\ldots,T_d\}$.
\end{lemma}

\begin{proof}
We present the proof in the case $d=2$. The argumentation for other dimensions is fully analogous. 
We follow the reasoning from \cite[Lemma 2]{DEB2} making modifications and skipping some details, which can be found in \cite{DEB2}.

Since $T_1(u)T_2(u)\sim \tau_1\tau_2 m(u)$, as $u\to\infty$, we get
\begin{equation}\label{22}
u^2 = 2\log (T_1T_2)
+ \left(\frac{2}{\alpha_1}+\frac{2}{\alpha_2}-1\right)\log\log (T_1T_2)
+ O(1).
\end{equation}

It is not difficult to see that there exists a constant $\delta\in(0,1)$ such that for all sufficiently large~$L$
$$\sup_{\varepsilon\leqslant \|\w{t}\|_\infty \leqslant L} \max\{|r(\w{t})|,\varrho_L(\w{t})\}<\delta.$$
Denote by $\beta$ a constant satisfying $0<\beta<(1-\delta)/(1+\delta)$ and divide the set $\mathcal{Q}:=[-T_1,T_1]\times[-T_2,T_2]-(-\varepsilon,\varepsilon)^2$ into two subsets:
\begin{eqnarray*}
\mathcal{S}^*\!\!&:=& \left\{\w{t}\in \mathcal{Q}: {|t_1|}\leqslant m(u)^{\beta / 2}, {|t_2|}\leqslant m(u)^{\beta / 2}\right\},\\
\mathcal{S} &:=& \mathcal{Q}  - \mathcal{S}^*.
\end{eqnarray*}
Observe that the shape of the set $\mathcal{S}^*$ of volume $m(u)^\beta(1+o(1))$ does not depend on the choice of $m_1$ and $m_2$.

Following line-by-line the arguments from \cite{DEB2}, thanks to the proper choice of~$\beta$, we obtain
\begin{equation}\label{AAA}
\frac{T_1T_2}{q_1q_2}\sum_{\w{jq}\in \mathcal{S}^*} \rho_{T }(\w{jq})\exp\left(-\frac{u^2}{1+\max\{|r(\w{jq})|,\varrho_{T }(\w{jq})\}}\right)\to 0,
\end{equation}
as $u\to\infty$.

To complete the proof, it suffices to show that
\begin{equation}\label{BBB}
\frac{T_1T_2}{q_1q_2}\sum_{\w{jq}\in \mathcal{S}} \rho_{T }(\w{jq})\exp\left(-\frac{u^2}{1+\max\{|r(\w{jq})|,\varrho_{T }(\w{jq})\}}\right)\to 0,
\end{equation}
as $u\to\infty$. By an argument from \cite{DEB2} and the fact that $m(u)^{\beta/2}\to\infty$, we get
$$\max\left\{|r(\w{jq})|, \varrho_{T }(\w{jq})\right\}\leqslant \frac{C}{\log m(u)^{\beta / 2}},$$
for sufficiently large $u$, some constant $C>0$ and all points $\w{jq}\in \mathcal{Q}$
satisfying $\|\w{jq}\|_\infty\geqslant m(u)^{\beta / 2}$.
Hence we have
\begin{eqnarray*}
\lefteqn{\frac{T_1T_2}{q_1q_2}\sum_{\w{jq}\in \mathcal{S}} \rho_{T }(\w{jq})
\exp\left(-\frac{u^2}{1+\max\left\{|r(\w{jq})|,\varrho_{T }(\w{jq})\right\}}\right)}\\
&\leqslant& {4}\frac{T_1^2T_2^2}{q_1^2q_2^2}\exp\left(-u^2\left(1-\frac{C}{\log m^{\beta / 2}}\right)\right)\frac{1}{\log m^{\beta / 2}}\\
&& \quad\quad\quad\quad\quad\quad \times \;
\frac{q_1q_2\log m^{\beta / 2}}{T_1T_2}\sum_{\w{jq}\in \mathcal{S}}\left|r(\w{jq})-\frac{R}{\log T }\right|\\
&=:&I_1(u)\times I_2(u).
\end{eqnarray*}

Applying the equality (\ref{22}), the definition of the functions $q_1$ and $q_2$ and the convergence $\log(T_1(u)T_2(u)) / \log m(u)^{\beta / 2} \to 2/\beta$, as $u\to\infty$,we conclude that $I_1$ is bounded. Our argumentation is analogous to the one given in \cite{DEB2}. The strong condition (\ref{OGR}) turns out not to be necessary.

In the next step we prove that $I_2(u)\to 0$ as $u\to\infty$. Observe that we have
\begin{eqnarray*}
I_2(u)
& \leqslant & \frac{q_1q_2}{T_1T_2}\sum_{\w{jq}\in \mathcal{S}}\left|r(\w{jq})\log\sqrt{(j_1q_1)^2+(j_2q_2)^2}-R\right| (1+o(1))\\
&& + \quad \beta R \frac{q_1q_2}{T_1T_2} \sum_{\w{jq}\in \mathcal{S}}\left|1-\frac{\log T }{\log\sqrt{(j_1q_1)^2+(j_2q_2)^2}}\right|(1+o(1))\\
&=:& J_1(u)+J_2(u).
\end{eqnarray*}
We need to show that both $J_1$ and $J_2$ tend to zero. Note that $J_1(u)\to 0$ as $u\to\infty$, due to \textbf{A3}.
Additionally, 
$$ J_2(u) \leqslant \frac{2R}{\log m}\frac{q_1q_2}{T_1T_2}
\sum_{\w{jq}\in \mathcal{S}} 
\left|\log\left(\frac{\sqrt{(j_1q_1)^2+(j_2q_2)^2}}{T }\right)\right|$$
and hence
\begin{eqnarray*}
J_2(u) = \frac{2R}{\log m} \cdot O\left(\bigintsss_{0}^1 \bigintsss_{0}^1 \left|\log(\sqrt{x^2+y^2})\right|dxdy
+ \bigintsss_{0}^1|\log |x||dx \right).
\end{eqnarray*}
Thus (\ref{BBB}) holds. The combination of (\ref{AAA}) and (\ref{BBB}) completes the proof.
\end{proof}

\subsection{Proof of Theorem \ref{TWDWYM}}\label{DWYM_DD}

To establish the main result, we develop the ideas given in \cite{LEAD,AREN,TAN, DEB}. The following proof of~Theorem \ref{TWDWYM} combines the method of proof of Theorem \ref{TW2WYM} for $d=2$ and $\gamma_1=\gamma_2=1/4$ (see \cite[Theorem 2]{DEB}), 
the lemmas from Section \ref{LEMMAS} and some new observations. 

The proof consists of two parts. In (i), we present a complete argumentation for the special case $\mathcal{J}=[0,1]^d$. In (ii), we explain how to apply the first part of the proof to obtain the limit theorem for arbitrary $\mathcal{J}$.

\noindent {\bf (i)} 
Let us consider $\mathcal{J}=[0,1]^d$. Then $\mathcal{J}^{\w{x}}_{\w{m}}=\mathcal{R}^{\w{x}}_{\w{m}}$ for $\w{x}\in(0,\infty)^d$.
Let $\{X^{\w{k}}(\w{t})\}$, for $\w{k}\in\mathbb{N}^{d-k}$, be independent copies of $\{X(\w{t})\}$ and let 
$$\eta(\w{t}):=X^{\w{k}(\w{t})}(\w{t}),\qquad \text{for} \quad \w{t}\in [0,\infty)^d,$$
with $\w{k}(\w{t}) = (\lfloor t_{k+1}\rfloor + 1, \lfloor t_{k+2}\rfloor + 1, \ldots, \lfloor t_d\rfloor + 1)$.
For any $T>0$, we define a~Gaussian random field $\{Y_T(\w{t}):\w{t}\in[0,T]^d\}$ as follows
\begin{equation*}
Y_T(\w{t}):=\left(1-\frac{R}{\log T}\right)^{1/2}\eta(\w{t})+\left(\frac{R}{\log T}\right)^{1/2}\mathcal{W},
\end{equation*}
where $\mathcal{W}$ denotes an $N(0,1)$ random variable independent of $\{\eta(\w{t})\}$. Then the covariance $C_T(\w{t},\w{t}+\w{s}):=\Cov(Y_T(\w{t}),Y_T(\w{t}+\w{s}))$ equals
\[ C_T(\w{t}, \w{t}+\w{s}) = \left\{ \begin{array}{ll}
r(\w{s})+(1-r(\w{s}))\frac{R}{\log T}
& ,\,  \text{if } \lfloor s_i+t_i\rfloor = \lfloor t_i\rfloor \text{ for }k < i \leqslant d;\\
\frac{R}{\log T}
&,\, \text{otherwise}. \end{array} \right.
\]

For $\w{x}\in (0,\infty)^d$ we define $\w{n}(\w{x},\w{m}):=(n^\w{x}_1,n^\w{x}_2,\ldots,n^\w{x}_d)$ with $n^\w{x}_i:=x_i M_i$ for $i\in\{1,2,\ldots,k\}$ and $n^\w{x}_i=n^\w{x}_i(u):=\left\lfloor x_i m_i(u) \right\rfloor$ for $i\in\{k+1,k+2,\ldots,d\}$. Since
\begin{eqnarray*}
\P\left(\sup_{\w{t}\in \mathcal{R}_\w{m}^{\w{x}}}X(\w{t})\leqslant u\right)
-  \P\left(\sup_{\w{t}\in \mathcal{R}^{\w{n}(\w{x},\w{m})}}X(\w{t})\leqslant u\right) = o(1),\qquad\text{as}\quad u\to\infty,
\end{eqnarray*}
we may focus on the asymptotics of the right-hand side of the above equality.

{\em Step 1.}  Let $\varepsilon>0$ be fixed. We divide the set $\mathcal{R}^{\w{n}(\w{x},\w{m})}$ into $n^\w{x}_{k+1} n^\w{x}_{k+2} \cdots n^\w{x}_d$ boxes 
$$\mathcal{G}_\w{l}:=\prod_{i=1}^{k}[0,x_iM_i] \times \prod_{i=k+1}^d [l_i-1,l_i],$$
indexed by $\w{l}=(l_{k+1}, l_{k+2},\ldots, l_d)\in\mathbb{N}^{d-k}$ such that $1\leqslant l_i\leqslant n_i^{\w{x}}$. Next we split each 
box $\mathcal{G}_\w{l}$ into two subsets $\mathcal{I}_{\w{l}}$ and $\mathcal{I}_{\w{l}}^*$ as follows
\begin{eqnarray*}
\mathcal{I}_{\w{l}}  &:=& \prod_{i=1}^{k}[0,x_iM_i] \times \prod_{i=k+1}^d [(l_i-1)+\varepsilon,l_i],\\
\mathcal{I}_{\w{l}}^*&:=&  \mathcal{G}_\w{l} - \mathcal{I}_{\w{l}}.
\end{eqnarray*}
To simplify the notation, we will write 
$$\mathcal{I}:=\bigcup \left\{\mathcal{I}_{\w{l}}: \w{1}\leqslant \w{l} \leqslant \left(n_{k+1}^{\w{x}}, n_{k+2}^{\w{x}}, \ldots, n_d^{\w{x}}\right) \right\}.$$
Applying the Bonferroni inequality, stationarity and the asymptotics (\ref{ASYMPTOTICS}), we get
\begin{eqnarray*}
\limsup_{u\to\infty}\left|\P\left(\sup_{\w{t}\in\mathcal{R}^{\w{n}(\w{x},\w{m})}}X(\w{t})\leqslant u\right)-\P\left(\sup_{\w{t}\in\mathcal{I}}X(\w{t})\leqslant u\right)\right|\\
\leqslant \limsup_{u\to\infty} n^\w{x}_{k+1} n^\w{x}_{k+2} \cdots n^\w{x}_d \P\left( \sup_{\w{t}\in \mathcal{I}^*_\w{1}} X(\w{t}) > u\right) \leqslant \zeta_1(\varepsilon),
\end{eqnarray*}
uniformly for $\w{x}\in [A,B]^d$, with $\zeta_1(\varepsilon)\to 0$ as $\varepsilon\to 0$.

{\em Step 2.} 
Let $a>0$ be fixed and let $q_1,q_2,\ldots,q_d$ be defined as at the begin of Section~\ref{LEMMAS}. Then we have
\begin{eqnarray*}
\lefteqn{\limsup_{u\to\infty} \left|\P\left(\sup_{\w{t}\in\mathcal{I}}X(\w{t})\leqslant u\right)
- \P\left(\sup_{\w{jq}\in\mathcal{I}}X(\w{jq})\leqslant u\right) \right| }\\
& \leqslant & \limsup_{u\to\infty} n^\w{x}_{k+1} n^\w{x}_{k+2} \cdots n^\w{x}_d  \left|\P\left(\sup_{\w{t}\in \mathcal{I}_\w{1}}X(\w{t})\leqslant u\right)
- \P\left(\sup_{\w{jq}\in \mathcal{I}_\w{1}}X(\w{jq})\leqslant u\right) \right| \\
&\leqslant & \zeta_2(a),
\end{eqnarray*}
uniformly for $\w{x}\in  [A,B]^d$, with $\zeta_2(a)\to 0$ as $a\to 0$, due to the Bonferroni inequality and Lemma \ref{LEMAT1}.

{\em Step 3.}
Let $T$ be a function defined as follows
$$T(u):= B \max \{m_1(u), m_2(u),\ldots, m_d(u)\}.$$ 
Note that if $T=T(u)$ is sufficiently large (and thus, if $u$ is sufficiently large), then
\begin{eqnarray*}
\left|r((\w{j}-\w{j'})\w{q})-C_T(\w{jq},\w{j'q})\right|
&\leqslant& \rho_T((\w{j}-\w{j'})\w{q}),\\
\left|C_T(\w{jq},\w{j'q})\right|
&\leqslant& \varrho_T((\w{j}-\w{j'})\w{q}),
\end{eqnarray*}
where the functions $\rho_T$ and $\varrho_T$ are defined by (\ref{FUN}). Moreover, for all pairs of points $\w{jq}, \w{j'q}\in \mathcal{I}$ satisfying $\|\w{j}-\w{j'}\|_\infty<\varepsilon$, provided that $\varepsilon$ is sufficiently small, we obtain
\begin{eqnarray*}
\left|r((\w{j}-\w{j'})\w{q})-C_T(\w{jq},\w{j'q})\right|&=& \frac{R\cdot(1-r((\w{j}-\w{j'})\w{q}))}{\log T},\\
\max\left\{\left|r((\w{j}-\w{j'})\w{q})\right|,\left|C_T(\w{jq},\w{j'q})\right|\right\} &= & r((\w{j}-\w{j'})\w{q})+\frac{R\cdot (1-r((\w{j}-\w{j'})\w{q}))}{\log T}.
\end{eqnarray*}
Combining the above properties, the normal comparison lemma \cite[Theorem 4.2.1]{LEAD} and Lemmas \ref{LEMACIK1} and \ref{LEMACIK2} in the same way as in \cite{DEB}, we conclude that
\begin{equation*}
\lim_{u\to\infty}\left|\P\left(\sup_{\w{jq}\in\mathcal{I}}X(\w{jq})\leqslant u\right) - \P\left(\sup_{\w{jq}\in\mathcal{I}}Y_T(\w{jq})\leqslant u\right)\right|= 0,
\end{equation*}
uniformly for $\w{x}\in [A,B]^d$.

{\em Step 4.} By the definition of the random field $\{Y_T(\w{t})\}$, we have
\begin{equation*}
\P\left(\sup_{\w{jq}\in\mathcal{I}}Y_T(\w{jq})\leqslant u\right)
=\bigints_{-\infty}^{\infty}\P\left(\eta(\w{jq}) \leqslant \frac{u-(R/\log T)^{1/2}z}{(1-R/\log T )^{1/2}} \, ;\, \w{jq}\in\mathcal{I} \right)d\Phi(z).
\end{equation*}
Since $T=T(u)=\exp(\gamma u^2)c(u)$ for some function $c$ satisfying $\log c(u)=o(u^2)$, the following condition
\begin{eqnarray*}
u_z&:=&\frac{u-(R/\log T)^{1/2}z}{(1-R/\log T )^{1/2}}\\
&=&\left(u-\left(\frac{R}{\log T}\right)^{1/2}z\right)\left( 1+\frac{R}{2\log T}+o\left(\frac{R}{\log T}\right)\right) \\
&=&u+ \frac{1}{u}\left(-\sqrt{\frac{R}{\gamma}}z+\frac{R}{2\gamma}\right)+o\left(\frac{1}{u}\right)
\end{eqnarray*}
holds for every $z\in \mathbb{R}$. Moreover, as $u\to\infty$,
$$\frac{m(u)}{m(u_z)}=\frac{u_z^{2/\alpha_1}u_z^{2/\alpha_2}\cdots u_z^{2/\alpha_d}\Psi(u_z)}{u^{2/\alpha_1}u^{2/\alpha_2}\cdots u^{2/\alpha_d}\Psi(u)}\to\exp\left(-\frac{R}{2\gamma}+\sqrt{\frac{R}{\gamma}}z\right)$$
and thus
\begin{equation}\label{WYK}
n_{k+1}^{\w{x}}n_{k+2}^{\w{x}}\cdots n_d^{\w{x}}=\frac{x_{k+1}\cdots x_d}{M_1 \cdots M_k}\exp\left(-\frac{R}{2\gamma}+\sqrt{\frac{R}{\gamma}}z\right) m(u_z) (1+o(1)).
\end{equation}
Applying the dependence structure of $\{\eta(\w{t})\}$ and stationarity of $\{X(\w{t})\}$, we obtain
\begin{eqnarray*}
\P\left(\sup_{{\w{jq}}\in\mathcal{I}}\eta(\w{jq})\leqslant u_z\right) &=&
\P\left(\sup_{\w{jq}\in \mathcal{I}_{\w{1}}}X(\w{jq})\leqslant u_z\right)^{n_{k+1}^{\w{x}}n_{k+2}^{\w{x}}\cdots n_d^{\w{x}}}+o(1).
\end{eqnarray*}
By Lemma \ref{LEMAT1}, the definition of $m(u_z)$ and properties (\ref{ASYMPTOTICS}) and (\ref{WYK}), we get
\begin{eqnarray*}
\lefteqn{\P\left(\sup_{\w{jq}\in \mathcal{I}_{\w{1}}}X(\w{jq})\leqslant u_z\right)^{n_{k+1}^{\w{x}}n_{k+2}^{\w{x}}\cdots n_d^{\w{x}}}}\\
&\leqslant &  \left(\P\left(\sup_{\w{t}\in \mathcal{G}_\w{1}}X(\w{t})\leqslant u_z\right)+\frac{\prod_{i=1}^k M_ix_i\cdot(\vartheta(a)\!+\!2\varepsilon\!+\!o(1))}{m(u_z)}\right)^{n_{k+1}^{\w{x}}n_{k+2}^{\w{x}}\cdots n_d^{\w{x}}} \\
&= & \left(1\!-\! \frac{\prod_{i=1}^k \! M_ix_i\cdot(1\!-\!\vartheta(a)\!-\!2\varepsilon\!+\!o(1))}{m(u_z)}\right)^{\frac{x_{k+1}\cdots x_d}{M_1 \cdots M_k}\exp\left(\!-\frac{R}{2\gamma}+\sqrt{\frac{R}{\gamma}}z \right) m(u_z\!)} \!\!\!\!\!\! +o(1)\\
&&\xrightarrow[u\to\infty]{} \exp\left(-\left(1-\vartheta(a)-2\varepsilon\right)x_1x_2\cdots x_d \exp\left(-\frac{R}{2\gamma}+\sqrt{\frac{R}{\gamma}}z\right)\right),
\end{eqnarray*}
where $\vartheta(a)\to 0$ as $a\to 0$. Thus
\begin{eqnarray*}
\lefteqn{\limsup_{u\to\infty}\bigints_{-\infty}^\infty \P\left(\sup_{\w{jq}\in \mathcal{I}_{\w{1}}}X(\w{jq})\leqslant u_z\right)^{n_{k+1}^{\w{x}}n_{k+2}^{\w{x}}\cdots n_d^{\w{x}}}d\Phi(z)}\\
&\leqslant& \Ex \exp\left(-\left(1-\vartheta(a)-2\varepsilon\right)x_1x_2\cdots x_d\exp\left(-\frac{R}{2\gamma}+\sqrt{\frac{R}{\gamma}}\mathcal{W}\right)\right).
\end{eqnarray*}
On the other hand, we have
\begin{eqnarray*}
\lefteqn{\P\left(\sup_{\w{jq}\in \mathcal{I}_{\w{1}}}X(\w{jq})\leqslant u_z\right)^{n_{k+1}^{\w{x}}n_{k+2}^{\w{x}}\cdots n_d^{\w{x}}}}\\
&\geqslant & \P\left(\sup_{\w{t}\in \mathcal{G}_\w{1}}X(\w{t})\leqslant u_z\right)^{n_{k+1}^{\w{x}}n_{k+2}^{\w{x}}\cdots n_d^{\w{x}}} \\
&\geqslant & \left(1-\frac{\prod_{i=1}^k M_ix_i}{m(u_z)}\right)^{\frac{x_{k+1}x_{k+2}\cdots x_d}{M_1M_2\cdots M_k}\exp\left(-\frac{R}{2\gamma}+\sqrt{\frac{R}{\gamma}}z\right) m(u_z)} \!\!\! +o(1)\\
&&\xrightarrow[u\to\infty]{} \exp\left(-x_1x_2\cdots x_d\exp\left(-\frac{R}{2\gamma}+\sqrt{\frac{R}{\gamma}}z\right)\right)
\end{eqnarray*}
and thus
\begin{eqnarray*}
\lefteqn{\liminf_{u\to\infty}\bigints_{-\infty}^\infty \P\left(\sup_{\w{jq}\in \mathcal{I}_{\w{1}}}X(\w{jq})\leqslant u_z\right)^{n_{k+1}^{\w{x}}n_{k+2}^{\w{x}}\cdots n_d^{\w{x}}}d\Phi(z)}\\
&\geqslant& \Ex \exp\left(-x_1x_2\cdots x_d\exp\left(-\frac{R}{2\gamma}+\sqrt{\frac{R}{\gamma}}\mathcal{W}\right)\right).
\end{eqnarray*}
Summarizing,
\begin{eqnarray}\label{Y_T_ASYMPTOTICS}
\lefteqn{\Ex \exp\left(-x_1x_2\cdots x_d\exp\left(-\frac{R}{2\gamma}+\sqrt{\frac{R}{\gamma}}\mathcal{W}\right)\right) } \nonumber\\
&\leqslant &\liminf_{u\to\infty}\P\left(\sup_{\w{jq}\in\mathcal{I}}Y_T(\w{jq})\leqslant u\right) 
\leqslant  \limsup_{u\to\infty}\P\left(\sup_{\w{jq}\in\mathcal{I}}Y_T(\w{jq})\leqslant u\right)\\
&\leqslant &\Ex \exp\left(-\left(1-\vartheta(a)-2\varepsilon\right)x_1x_2\cdots x_d\exp\left(-\frac{R}{2\gamma}+\sqrt{\frac{R}{\gamma}}\mathcal{W}\right)\right),\nonumber
\end{eqnarray}
uniformly for $\w{x}\in[A,B]^d$.

{\em Step 5.} Form the steps 1-3 of the proof we know that
\begin{equation*}
\limsup_{u\to\infty} \left|\P\left(\sup_{\w{t}\in\mathcal{R}^{\w{n}(\w{x},\w{m})}}X(\w{t})\leqslant u\right) - \P\left(\sup_{\w{jq}\in\mathcal{I}}Y_T(\w{jq})\leqslant u\right)\right|\leqslant \zeta_1(\varepsilon) + \zeta_2(a),
\end{equation*}
uniformly for $\w{x}\in [A,B]^d$, with $\zeta_1(\varepsilon)\to 0$ as $\varepsilon\to 0$ and $\zeta_2(a)\to 0$ as $a\to 0$.
Combining it with the inequalities (\ref{Y_T_ASYMPTOTICS}) and passing with $\varepsilon\to 0$ and $a\to 0$, we finish the first part of the proof.

\noindent{\bf  (ii)} Let $\mathcal{J}\subset\mathbb{R}^d$ be an arbitrary Jordan-measurable set with Lebesgue measure $\lambda(\mathcal{J})>0$. We follow the argumentation from \cite[Theorem 2 (ii)]{DEB}. Observe that for every $\varepsilon>0$, there exist some positive constants $z_1,z_2,\ldots,z_d$ and some sets $\mathcal{L}_\varepsilon,\mathcal{U}_\varepsilon\subset \mathbb{R}^d$ being finite sums of disjoint closed
hyperrectangles with dimensions $z_1\times z_2 \times \cdots \times z_d$, such that
 $\mathcal{L}_\varepsilon\subset \mathcal{J} \subset \mathcal{U}_\varepsilon$
and $\lambda(\mathcal{L}_\varepsilon)+\varepsilon>\lambda(\mathcal{J})>\lambda(\mathcal{U}_\varepsilon)-\varepsilon$. Then, following nearly line-by-line the arguments given in the proof of~part (i), we obtain
\[
\P\left(\sup_{\w{t}\in (\mathcal{L}_{\varepsilon})^{\w{x}}_\w{m}}\!\! X(\w{t})\leqslant u\right)
\to
\Ex\exp\left(- x_1x_2\cdots x_d\lambda (\mathcal{L}_{\varepsilon})\exp\left(-\frac{R}{2\gamma}+\sqrt{\frac{R}{\gamma}}\mathcal{W}\right)\right)
\]
and
\[
\P\left(\sup_{\w{t}\in (\mathcal{U}_{\varepsilon})^{\w{x}}_\w{m}}\!\! X(\w{t})\leqslant u\right)
\to
\Ex\exp\left(- x_1x_2\cdots x_d\lambda (\mathcal{U}_{\varepsilon})\exp\left(-\frac{R}{2\gamma}+\sqrt{\frac{R}{\gamma}}\mathcal{W}\right)\right),\]
as $u\to\infty$, uniformly for $\w{x}\in[A,B]^d$.
Since $\varepsilon>0$ is arbitrarily small, it gives
\[
\P\left(\sup_{\w{t}\in \mathcal{J}^{\w{x}}_\w{m}}X(\w{t})\leqslant u\right)
\to
\Ex\exp\left(- x_1x_2\cdots x_d\lambda (\mathcal{J})\exp\left(-\frac{R}{2\gamma}+\sqrt{\frac{R}{\gamma}}\mathcal{W}\right)\right),
\]
as $u\to\infty$, uniformly for $\w{x}\in[A,B]^d$, which finishes the proof.

\vspace{0.5cm}
\noindent\textbf{Acknowledgements.}
The author would like to thank Krzysztof D\ke{}bicki and Micha{\l} Kukie{\l}a for comments and suggestions.

\end{document}